\documentclass{amsart} 
\usepackage{amscd,amssymb,graphics}

\usepackage{amsfonts}
\usepackage{amsmath}
\usepackage{amsxtra}
\usepackage{latexsym}
\usepackage[mathcal]{eucal}

\input xy
\xyoption{all}
\usepackage{epsfig}
\usepackage[hidelinks]{hyperref}

\newtheorem{thm}{Theorem}[section]

\newtheorem{cor}[thm]{Corollary}
\newtheorem{lem}[thm]{Lemma}

\theoremstyle{definition}
\newtheorem{defin}[thm]{Definition}

\theoremstyle{remark}

\newtheorem{remarks}[thm]{Remarks}

\newtheorem{exs}[thm]{Examples}

\numberwithin{equation}{section}

\newcommand{\delete}[1]{} 
\newcommand{\nt}{\noindent}

\def\eps{{\varepsilon}}

\newcommand{\ben}{\begin{enumerate}}

\newcommand{\een}{\end{enumerate}}
\newcommand{\bit}{\begin{itemize}}

\newcommand{\eit}{\end{itemize}}

\def\R {{\mathbb R}}
\def\N {{\mathbb N}}

\def\Homeo{{\mathrm{Homeo}}\,}

\def\B{{\mathcal{B}}}

\def\F{{\mathcal F}}

\def\QED{\nobreak\quad\ifmmode\roman{Q.E.D.}\else{\rm Q.E.D.}\fi}

\def\a{\alpha}

\def\g{\gamma}

\newcommand{\cls}{{\rm{cls\,}}}

\newcommand{\sgn}{{\rm{sgn}}}

\begin{document}

\title[]
{A note on tameness of families having bounded variation}


\author[]{Michael Megrelishvili}
\address{Department of Mathematics,
Bar-Ilan University, 52900 Ramat-Gan, Israel}
\email{megereli@math.biu.ac.il}
\urladdr{http://www.math.biu.ac.il/$^\sim$megereli}



\thanks{This research was supported by a grant of Israel Science Foundation (ISF 668/13)}

\date{December 11, 2016}

\begin{abstract} We show that for arbitrary linearly ordered set $(X, \leq)$ any bounded family of (not necessarily, continuous) real valued functions on $X$ with bounded total variation does not contain independent sequences. We obtain generalized Helly's sequential compactness type theorems. One of the theorems asserts that for every compact metric space $(Y,d)$ the compact space $BV_r(X,Y)$ of all functions $X \to Y$ with variation $\leq r$ is sequentially compact in the pointwise topology. Another Helly type theorem shows that the compact space $M_+(X,Y)$ of all order preserving maps $X \to Y$ is sequentially compact where $Y$ is a compact metrizable partially ordered space in the sense of Nachbin. 
\end{abstract}

\maketitle


\section{Introduction} 
  
  Recall that the 
  Helly's compact space $M_+([0,1],[0,1])$ of all increasing selfmaps on the closed unit interval $[0,1]$ is sequentially compact in the pointwise topology.  
  A slightly more general form of this result is the following classical result of Helly (see \cite{Helly} and also \cite{Natanson}). 
  
  \begin{thm} \label{t:HellyClassic} 
  {\bf (Helly's selection theorem)} 
  	For every sequence of functions from the set $BV_r([a,b],[c,d])$ of all real functions $[a,b] \to [c,d]$ with variation $\leq r$ there exists a pointwise convergent subsequence. 
  	That is, $BV_r([a,b],[c,d])$ is sequentially compact.
 \end{thm}
  
  There are several generalized forms of Helly's theorem in the literature. Among other relevant references we mention \cite{FP, BC, Ch}.  
  In Section \ref{s:H} we give 
 two generalized versions of Helly's theorem for functions defined on abstract linearly ordered sets. Namely, Theorems \ref{corOfGenHelly1} and \ref{corOfGenHelly2} which are partial generalizations of Fuchino-Plewik \cite[Theorem 7]{FP} 
  (full generalization under $\mathbf{s}=\aleph_1$ where $\mathbf{s}$ denotes the \emph{splitting number}  \cite{FP}, in particular this is the case under the Continuum Hypothesis)  
 and Belov-Chistyakov \cite[Theorem 1]{BC}. 
 	
 One of the main ideas in our approach is the independence 
 for families of real valued functions on a set $X$. This concept plays a major role in several research lines. For example, in Rosenthal's $l^1$-theorem, in Bourgain-Fremlin-Talagrand dichotomy and related topics, \cite{Ro,BFT,Tal, Dulst}. 
A relatively new direction when the (non)independence of families becomes very important is dynamical systems theory. Especially, tame systems and tame coding sequences. See \cite{Ko,Gl-tame,KL, GM-survey, GM-tame,GM-circular} and  references therein. 

We give 
a sufficient condition under which a given bounded family $F$ of real functions on a linearly ordered set $X$ is \emph{tame}, i.e., does not contain any independent sequence. We show in Theorem \ref{newprinciple} 
that this happens for example when $F$ has a bounded total variation. 
This is easy in the particular case 
 when every member $f \in F$ is an order preserving function (Example \ref{ex:tame}.6). 
Another sufficient condition for the tameness of $F$ (for arbitrary set $X$) is the Grothendieck's double limit property (Example \ref{ex:tame}.3). 

We use a topological characterization of independent families of continuous functions on compact spaces, Theorem \ref{f:sub-fr}. It is a reformulation of a result 
presented in van Dulst's book \cite[Theorem 3.11]{Dulst} which can be traced back to results of Rosenthal \cite{Ro,Ros1} and Bourgain-Fremlin-Talagrand \cite{BFT,Tal}. It asserts that a family $F$ of bounded continuous functions on a compact space is tame iff each sequence in $F$ has a pointwise convergent subsequence in $\R^X$ iff the pointwise closure  $\cls (F)$ consists of the functions with the point of continuity property. It is equivalent to saying that each member of $\cls (F)$ is a \emph{fragmented} function (Definition \ref{def:fr}). This motivates Theorem \ref{monot}: every order preserving function on every linearly ordered set is fragmented. 
Next we deal with functions of bounded variation defined on abstract ordered spaces. 
By an analog of Jordan's decomposition (Lemma \ref{l:BVprop}.3) every function of a bounded variation is fragmented. 
Using results of Nachbin on ordered compactifications we give a representation theorem \ref{RepLem} which, as Theorem \ref{monot}, hopefully, has an independent interest.  

Some dynamical applications of Theorems \ref{monot} and \ref{newprinciple}  are presented in
 \cite{GM-tame,GM-circular}, 
 where we show that several important coding functions (for example, multidimensional Sturmian bisequences and finite coloring functions on the circle) 
on dynamical $G$-systems $X$ 
lead to functions $f: X \to \R$ the $G$-orbit $fG$ 
of which are tame families. 

 \vskip 0.2cm
\section{Fragmentability and independence} 

By $\cls$ we denote the closure operator. 
We use the usual definition of uniform structures using 
the entourages. 
We allow not necessarily Hausdorff uniform spaces. 
So, involving, in particular, the uniform structures induced by a pseudometric.  
``Compact" will mean ``compact and Hausdorff". Recall that any compact space $X$ admits a unique compatible uniform structure. Namely the set of all neighborhoods of the diagonal in $X \times X$.

 \vskip 0.2cm 
 
\subsection{Fragmented maps} 

\begin{defin} \label{def:fr}
	\cite{JOPV} Let $(X,\tau)$ be a topological space and
	$(Y,\mu)$ a uniform space. We say that a function $f: X \to Y$ is 
	{\em fragmented\/} 
	if for every nonempty subset $A$ of $X$ and every entourage $\eps
	\in \mu$ there exists an open subset $O$ of $X$ such that $O \cap
	A$ is nonempty and the set $f(O \cap A)$ is $\eps$-small in $Y$. 
	Notation: $f \in {\mathcal F}(X,Y)$, whenever the
	uniformity $\mu$ is understood. If $Y=\R$ then we write simply
	${\mathcal F}(X)$.
\end{defin}
 
A function $f: X \to Y$ has the {\em point of continuity property}  
if for every closed nonempty subset $A$ of $X$ the restriction
$f|_A: A \to Y$ has a point of continuity. 
For compact $X$ and
(pseudo)metric space $(Y,d)$ it is equivalent to the fragmentability. 
If $X$ is Polish and $(Y,d)$ is a separable metric space then
$f: X \to Y$ is fragmented iff $f$ is a Baire class 1 function (i.e., the inverse image under $f$ of every open set is $F_\sigma$), \cite{Kech, GM-rose}. 

The topological concept of fragmentability comes from
Banach space theory. 
More facts about fragmentability see for example in \cite{N, JR, JOPV, me-fr, Me-nz, GM-rose, GM-tame}.

\vskip 0.2cm 
\subsection{Independent sequences of functions}
\label{s:ind}

	Let $f_n: X \to \R$ be a uniformly bounded sequence of functions on a \emph{set} $X$. Following Rosenthal \cite{Ro} we say that
	this sequence is an \emph{$l_1$-sequence} on $X$ if there exists a real constant $a >0$
	such that for all $n \in \N$ and choices of real scalars $c_1, \dots, c_n$ we have
	$$
	a \cdot \sum_{i=1}^n |c_i| \leq ||\sum_{i=1}^n c_i f_i||_{\infty}.
	$$

A Banach space $V$ is said to be {\em Rosenthal} if it does
not contain an isomorphic copy of $l_1$, or equivalently, if $V$ does not contain a sequence which is equivalent to an $l_1$-sequence.

	A sequence $f_n$ of	real valued functions on a set
	$X$ is said to be \emph{independent} (see \cite{Ro,Tal,Dulst}) if
	there exist real numbers $a < b$ such that
	$$
	\bigcap_{n \in P} f_n^{-1}(-\infty,a) \cap  \bigcap_{n \in M} f_n^{-1}(b,\infty) \neq \emptyset
	$$
	for all finite disjoint subsets $P, M$ of $\N$.

\begin{defin} \label{d:tameF} 
Let us say that a family $F$ of real valued (not necessarily, continuous)  functions on a set $X$ is {\it tame} if $F$ does not contain an independent sequence.
\end{defin}

 Such families play a major role in the theory of tame dynamical systems. See, for example, \cite{Ko,KL,GM-rose,GM-survey,GM-tame}. 

The following useful result is a reformulation of some known results.
It is based on results of Rosenthal \cite{Ro}, Talagrand \cite[Theorem 14.1.7]{Tal} and van Dulst \cite{Dulst}. See also \cite[Sect. 4]{GM-rose}.

\begin{thm} \label{f:sub-fr} \cite[Theorem 3.11]{Dulst} 
	Let $X$ be a compact space and $F \subset C(X)$ a bounded subset.
	The following conditions are equivalent:
	\begin{enumerate}
		\item
		$F$ does not contain an $l_1$-sequence. 
			\item $F$ is a tame family (does not contain an independent sequence). 
		\item
		Each sequence in $F$ has a pointwise convergent subsequence in $\R^X$.
		\item 
	The pointwise closure ${\cls}(F)$ of $F$ in $\R^X$ consists of fragmented maps,
	that is,
	${\cls}(F) \subset {\mathcal F}(X).$
	\end{enumerate}
\end{thm}

Let $X$ be a topological space and $F \subset l_{\infty}(X)$ be a norm bounded family. 
Recall that $F$ has Grothendieck's {\em Double Limit Property} (DLP) on $X$ if for every sequence $\{f_n\} \subset F$ and every sequence
$\{x_m\} \subset X$ the limits 
$$\lim_n \lim_m f_n(x_m) \ \ \  \text{and} \ \ \ \lim_m \lim_n f_n(x_m)$$
are equal whenever they both exist.

\begin{exs} \label{ex:tame} \ 
\ben  
\item A Banach space $V$ is Rosenthal iff 
every bounded subset $F \subset V$ is tame  (as a family of functions) on every bounded subset $X \subset V^*$ of the dual $V^*$.  
\item A Banach space is reflexive iff 
every bounded subset $F \subset V$ has DLP on every bounded subset 
$X \subset V^*$.
\item  ((DLP) $\Rightarrow$ Tame)  
Let $F$ be a bounded family of real valued (not necessarily, continuous) functions on a set $X$ such that $F$ has DLP. Then $F$ is tame.  
\item
The family $\Homeo [0,1]$, of all autohomeomorphisms of $[0,1]$, 
is tame (but not with DLP on $[0,1]$). 
\item 
The sequence of projections on the Cantor cube $$\{\pi_m: \{0,1\}^{\N} \to \{0,1\}\}_{m \in \N}$$ 
and the sequence of Rademacher functions
$$r_n: [0,1] \to \R, \ \ r_n(x):=\sgn (\sin (2^n \pi x))$$
both are independent (hence, nontame).  
\item Let $(X,\leq)$ be a linearly ordered set. Then any family $F$ of order preserving real functions is tame. Moreover there is no independent pair of functions in $F$. 
\een
\end{exs}
	\begin{proof}
		(1) Apply Theorem \ref{f:sub-fr} assuming $X=B_{V^*}$ is the weak$^*$  compact unit ball of the dual $V^*$. 	
			
		 (2) Use Grothendieck's double limit characterization 
		 of weak compactness (see for example \cite[Theorem A5]{BJM}) 
		 and a well known fact that a Banach space $V$ is reflexive iff its closed unit ball $B_V$ is weakly compact.   
		 
		 (3) One may suppose that $X$ is a dense subset of a compact space $Y$ and $F \subset C(Y)$. Indeed, take for example the maximal compactification $Y:=\beta X$ of the discrete copy of $X$. 
		 Since $F$ has DLP on $X$ we may apply \cite[Appendix A4]{BJM} which 
		 yields that the pointwise closure $\cls(F)$ of $F$ in $\R^Y$ is a subset of $C(Y)$. Now Theorem \ref{f:sub-fr} ((4) $\Rightarrow$ (2)) shows that $F$ is tame on $Y$ and hence also on $X \subset Y$.  
		 
		
		(4) Every homeomorphism $[0,1] \to [0,1]$ is either order preserving or order reversing. Now,  combine Helly's Theorem \ref{t:HellyClassic} and Theorem  \ref{f:sub-fr}  ((3) $\Rightarrow$ (2)).

		(5) These two examples are well known, \cite{Dulst}. 
		
		(6) Assuming that $f_1, f_2 \in F$ is an independent pair there exist $a < b$ and $x,y \in X$ such that $x \in f_1^{-1}(-\infty,a) \cap f_2^{-1}(b, \infty)$ and $y \in f_2^{-1}(-\infty,a) \cap f_1^{-1}(b, \infty)$. Then $f_1(x) < f_1(y)$ and $f_2(y) < f_2(x)$. Since $f_1$ and $f_2$ are order preserving and $X$ is linearly ordered we obtain that $x <y$ and $y < x$, a contradiction. 
	\end{proof}
	
 Note that in (1) and (2) the converse statements are true; as it follows from results of \cite{GM-tame} every tame (with DLP) family $F$ on $X$ can be represented, in a sense, on a Rosenthal (resp., reflexive) Banach space. 
Namely, there exist: a Rosenthal (resp., reflexive) space $V$, a pair $(\nu,\a)$ of bounded maps $\nu: F \to V, \ \alpha: X \to V^*$ such that 
$$
f(x)= \langle \nu(f), \a(x) \rangle
\ \ \ \forall \ f \in F, \ \ \forall \ x \in X.
$$
In other words, the following diagram commutes

$$\xymatrix{ F \ar@<-2ex>[d]_{\nu} \times X
	\ar@<2ex>[d]^{\a} \ar[r]  & \R \ar[d]^{id } \\
	V \times V^* \ar[r]  &  \R }
$$
 

\vskip 0.2cm 
\section{Order preserving maps}

{\it Partial order} will mean a reflexive, antisymmetric and transitive relation.

\begin{defin} \label{d:ord} \
	(Nachbin \cite{Nach})
	Let $(X,\tau)$ be a topological space and $\leq$ 
	a partial order on the set $X$. The triple 
	$(X,\tau,\leq)$ is said to be a
	\emph{compact (partially) ordered space} if $(X,\tau)$ is a compact space and the graph of the relation $\leq$ is $\tau$-closed in $X \times X$.
\end{defin}

Recall that for every linearly ordered set $(X,\leq)$ the rays $(a,\to)$, $(\leftarrow,b)$ with $a,b \in X$ form a subbase for the
standard \emph{interval topology} $\tau_{\leq}$ on $X$. 
The triple $(X,\tau_{\leq},\leq)$ is said to be a \emph{linearly ordered topological space}
(LOTS). 
Sometimes we write just $(X,\leq)$, or even simply $X$, where no ambiguity can occur.

\begin{lem} \label{ordHausd}
Let $(X, \leq)$ be a LOTS. Then for any
 two distinct points $u_1 < u_2$ in $X$
there exist disjoint $\tau_{\leq}$-open neighborhoods
$O_1$ and $O_2$ in $X$ of $u_1$ and $u_2$ respectively such that $O_1 < O_2$,
meaning that $x < y$ for every $(y,x) \in O_2 \times O_1$. In particular, the graph of $\leq$ is closed in $(X,\tau_{\leq}) \times (X,\tau_{\leq})$.
\end{lem}
\begin{proof} 
	If the interval $(u_1,u_2)$ is empty then take $O_1:=(\leftarrow,u_2)$ and $O_2:=(u_1,\rightarrow)$. If $(u_1,u_2)$ is nonempty then choose $t \in (u_1,u_2)$ and define $O_1:=(\leftarrow,t)$,  $O_2:=(t,\rightarrow)$.
	
	\end{proof}

 \begin{cor} \label{c:GLOTS} 
 Any \emph{compact} LOTS is a
compact ordered space in the sense of Nachbin (Definition \ref{d:ord}). 
Conversely, for every compact ordered space $(X,\tau,\leq)$, where $\leq$ is a linear order, necessarily $\tau$ is the interval topology of $\leq$. 
 \end{cor}
 \begin{proof}
 The first part is obvious by Lemma \ref{ordHausd}. 
 For the second part observe that the $\tau$-closedness of the linear order $\leq$ in $X \times X$ 
 implies that the subbase intervals $(a,\to)$, $(\leftarrow,b)$ (with $a,b \in X$) are $\tau$-open. Whence, $\tau_{\leq} \subseteq \tau$. Since $\tau_{\leq}$ is a Hausdorff topology and $\tau$ is a compact topology we can conclude that $\tau_{\leq} = \tau$. 
 
 \end{proof}
 
 A map $f: (X,\leq) \to (Y,\leq)$ between two (partially) ordered sets is said to be \emph{order preserving} or \textit{increasing} 
if $x \leq x'$ implies $f(x) \leq f(x')$ for every $x,x' \in X$.
 

Let $(X,\leq)$ and $(Y,\leq)$ be partially ordered sets.  
Denote by $M_+(X,Y)$ the set of all order preserving maps $X \to Y$. 
For $Y=\R$ we use the symbol 
$M_+(X,\leq)$ or $M_+(X)$. 
Since the order of $\R$ is closed in $\R^2$, we have $\cls(M_+(X)) = M_+(X)$. That is, $M_+(X)$ is pointwise closed in $\R^X$. 
If $(Y,\tau,\leq)$ is a compact partially ordered space then 
$M_+(X,Y)$ is pointwise closed in $Y^X$. 
For compact partially ordered spaces $X,Y$ we define also $C_+(X,Y)$ the set 
of all continuous and increasing maps $X \to Y$. 

Fundamental results of Nachbin \cite[p. 48 and 113]{Nach} imply the following

\begin{lem} \label{l:Nachbin} \emph{(Nachbin \cite{Nach})} 
	Let $(Y,\tau,\leq)$ be a compact partially ordered space (Definition \ref{d:ord}). Then the set $C_+(Y,[0,1])$  separates points of $Y$. 
\end{lem}
 
The following Theorem is a slightly 
generalized version of a recent result from \cite{GM-tame}. 
Its prototype is a well known fact that every monotonic function $[a,b] \to \R$ is a Baire 1 function. 

\begin{thm} \label{monot} \cite{GM-tame} Let $ (X,\leq)$ be a linearly ordered set and 
$(Y,\tau,\leq)$  
a compact partially ordered space. Then  
 every order preserving map $f: X \to (Y,\mu)$ is fragmented, where $\mu$ is the unique compatible uniformity on the compact space $Y$ and $X$ carries the interval topology. 
 \end{thm}
 \begin{proof}  
 	First note that the question can be reduced to the case of $Y:=[0,1]$. Indeed, 
 Lemma  \ref{l:Nachbin} implies that
there exists a point separating family  $\{q_i: Y \to [0,1]\}_{i \in I}$ of order preserving \emph{continuous}  maps.
Clearly the composition of two order preserving maps is order preserving.
Now by \cite[Lemma 2.3.3]{GM-rose} 
it is enough to show that every map $q_i \circ f$ is fragmented.
So we can assume that our order preserving function is 
of the form $f: X \to Y=[0,1]$.
We have to show that $f$ is fragmented. 

Now observe that one may assume that $X$ is compact. Indeed, for every LOTS $X$ (with its interval topology) there exists a compact LOTS $Z$ and an embedding of topological spaces and ordered sets $i: X \hookrightarrow Z$ (see for example, \cite[Exercise 3.12.3]{Eng}). 
Now define 
$$F: Z \to [0,1], \ F(z):=\sup \{f(x): x \in X, x \leq z\}$$
 for every $z \in Z$ when $\{x \in X: x \leq z\}$ is nonempty and 
 $F(z):=0$ if $\{x \in X: x \leq z\}$ is empty. 
Then $F$ is a well defined increasing function on $Z$ which
extends $f: X \to [0,1]$. Note that the fragmentability is a hereditary property. 
The fragmentability of $F: Z \to [0,1]$ guarantees the fragmentability of $f: X \to [0,1]$. So, below we assume that $X$ is compact.  

Assume the contrary that  $f: X \to Y=[0,1]$ is not fragmented. 
Then by \cite[Lemma 3.7]{Dulst} (using that $X$ is compact) 
there exists a closed subset $K \subset X$ and $a <b$ in $\R$ such that $$K \cap \{x\in X: \ f(x) \leq a\}, \ \ \ K \cap \{x\in X: \ b \leq f(x)\}$$
are both dense in $K$.

Choose arbitrarily two distinct points $k_1 < k_2$ in $K$.
By Lemma \ref{ordHausd} one can choose disjoint open neighborhoods
$O_1$ and $O_2$ in $X$ of $k_1$ and $k_2$ respectively such that $O_1 < O_2$.

By our assumption we can choose $x \in O_1 \cap K$ such that $b \leq f(x)$. Similarly, there exists
$y \in O_2 \cap K$ such that $f(y) \leq a$. Since $a<b$ we obtain
$f(y) < f(x)$. On the other hand, $x < y$ (because $O_1 < O_2$),
contradicting our assumption that $f$ is order preserving.

\end{proof}


The following result is an adaptation of some well known facts from the theory of ordered compactifications (see for example Fedorchuk \cite{Fed}, or Kaufman \cite{Kau}). 
We consider not necessarily continuous ``compactification" $\nu: X \to Y$ of a linearly ordered set $X$ 
as an increasing map into a compact LOTS $Y$. 
This is equivalent to saying that we consider order compactifications $X \to Y$ 
 of the discrete copy of $X$ (we do not require 
topological embeddability for compactification maps). 

\begin{thm} \label{RepLem} \emph{(Representation theorem)} 
	Let $(X,\leq)$ be a linearly ordered set. For any family 
	$\Gamma:=\{f_i: X \to [c,d]\}_{i \in I}$ (with $c <d$) of 
order preserving (not necessarily continuous) functions 
 there exist: a compact LOTS $(Y,\leq)$, an order preserving dense  
		injection $\nu:X \hookrightarrow Y$ and a family $\{F_i: Y \to [c,d]\}_{i \in I}$ of $\tau_{\leq}$-continuous increasing functions such that $f_i=F_i \circ \nu$ $\forall i \in I$. 
\end{thm}
\begin{proof} 
	For simplicity we assume that $[c,d]=[0,1]$. 
	Without restriction of generality one may assume that 
	$\Gamma = M_+(X,[0,1])$. So,  
	$\Gamma$ separates points of $X$. Indeed, 
	for every $a <b$ in $X$ consider the characteristic function $\chi_A: X \to [0,1]$ of $A:=\{x: b \leq x \}$. Then $\chi_A \in M_+(X,[0,1])$ and 
	separates $a$ and $b$.  
	Consider the diagonal map 
	$$\nu: X \to Y \subset [0,1]^I, \ \ \nu(x)(i)=f_i(x).$$ 
	Since $\Gamma$ separates the points, $\nu$ is an injection. 
	We will identify $X$ and the dense subset $\nu(X)$ in the compactum $Y:=\cls(\nu(X))$.  
	Let us show that $Y$ admits a naturally defined linear order which extends the order of $\nu(X)=X$. Consider the natural partial order $\g$ on $[0,1]^I$ 
	$$
	u \leq v \Leftrightarrow u_i \leq v_i \ \ \forall i \in I.
	$$ 
	It is easy to see that $\g$ is a partial order. 
	Clearly, it induces the original order on $X \subset [0,1]^I$.
	Indeed, if $x \leq x'$ in $X$ then 
	$x_i=f_i(x) \leq x_i'=f_i(x')$ for every $i \in I$ because each $f_i$ is increasing. So, we obtain that $(x,x') \in \g$. 
	Conversely, if $(x,x') \in \g$ and $x \neq x'$ then $f_i(x) \leq f_i(x')$ for every $i \in I$. Since $\Gamma$ (by our assumption) separates the points we obtain that $f_i(x) < f_i(x')$ 
	for some $i \in I$. 
	Since the order in $X$ is linear and $f_i$ is increasing we necessarily have $x <x'$. 
	
	 	\vskip 0.2cm 
	
	\noindent \textbf{Claim 1:} $\g \subset [0,1]^I \times [0,1]^I$ is a closed partial order on $[0,1]^I$.
	\vskip 0.2cm 
 We show that $\g$ is closed. Let $(u,v) \notin \g$. By definition this means that there exists $i \in I$ such that 
$u_i > v_i$ in $[0,1]$. Choose disjoint open (in $[0,1]$) 
intervals $U$ and $V$ of $u_i$ and $v_i$. Consider the basic open neighborhoods 
$A:=\pi_i^{-1}(U)$ and $B:=\pi_i^{-1}(V)$, where $\pi_i:[0,1]^I \to [0,1]$ is the $i$-th projection. Then $A \times B$ is a neighborhood of the point $(u,v)$ in $[0,1]^I \times [0,1]^I$ such that for every $(a,b) \in A \times B$ we have $a_i > b_i$. Hence, $(a,b) \notin \g$.     
	
	\vskip 0.2cm
	\noindent 	\textbf{Claim 2:} The restriction of $\g$ on $Y=\cls(X)$ is a linear order. 
	
	\vskip 0.2cm 
	
	Indeed, let $u,v$ be distinct 
	elements of $Y$. Then there exists $i \in I$ such that $u_i \neq v_i$. Say, 
	$$u_i=f_i(u) < v_i=f_i(v).$$ 
	Choose a real number $c$ 
	such that $f_i(u) < c < f_i(v)$ and define two open neighborhoods   
	$$
	u \in A:=\{y \in Y: y_i < c\} \ \ \ v \in B:=\{y \in Y: c < y_i\}
	$$
	of $u$ and $v$ in $Y$. Since $X$ is dense in $Y$ and $A, B$ are open subsets in $Y$, we have $A \subset \cls(A \cap X)$ and $B \subset \cls(B \cap X)$. 
	For every $a \in A \cap X, \ b \in B \cap X$ we obviously have 
	$$
	f_i(a) < c < f_i(b).
	$$
	Then necessarily, $a < b$ because the order on $X$ is linear and $f_i$ is increasing. 
	Since such $a$ approximates $u \in A$ and $b$ approximates $v \in B$ we obtain by Claim 1 that $u \leq v$. 
	
	\vskip 0.2cm
	
	\noindent 	\textbf{Claim 3:} Every restricted projection $F_i: Y \to [0,1]$   
	is a continuous (regarding the product topology) $\gamma$-increasing function and $F_i|_X=f_i$ $\forall i \in I$. 
	
	\vskip 0.2cm
	
	$F_i$ is increasing by definition of $\g$. The continuity is trivial by definition of the product topology.   
	Finally, $f_i$ is a restriction of $F_i$ on $X$ by the definition of diagonal map. 
	
	\vskip 0.2cm
	
	Now it is enough to show the last claim. 
	\vskip 0.2cm
	
	\noindent 	\textbf{Claim 4:} The product topology $\tau$ on $Y$ coincides with the interval topology $\tau_{\g}$. 
	
	\vskip 0.2cm
	This follows from the second part of Corollary \ref{c:GLOTS} taking into account that the linear order of $Y$ is closed in $Y \times Y$
	regarding the product topology (use Claim 1 and the closedness of $Y$ in  $[0,1]^I$).  
	
\end{proof}

\section{Functions of bounded variation on an ordered set} 

In the following definition we consider a natural generalization of the classical concept (well known for the interval $X=[a,b]$) of bounded variation.

\begin{defin} \label{d:BV} 
	Let $(X, \leq)$ be a linearly ordered set. We say that a bounded real valued function $f: X \to \R$ has variation not greater than $r$ if 
	\begin{equation} \label{BV1}
	\sum_{i=0}^{n-1} |f(x_{i+1})-f(x_i)| \leq r
	\end{equation} 
	for every choice of $x_0 \leq x_1 \leq \cdots \leq x_n$ in $X$. 
	The least upper bound of all such possible sums is the  {\it variation} of $f$. Notation:  $\Upsilon(f)$. If  $\Upsilon(f) \leq r$ then we write $f \in BV_r(X)$. If $f(X) \subset [c,d]$ for some $c \leq d$ then we write also $f \in BV_r(X,[c,d])$. One more notation: $BV(X):=\cup_{r>0} BV_r(X)$. 
\end{defin}


\begin{lem} \label{l:BVprop} 
	Let $(X, \leq)$ be a linearly ordered set.
	\ben 
	\item 
	$BV_r(X,[c,d])$ is a pointwise closed (hence, compact) subset of $[c,d]^X$.
	\item $M_+(X,[c,d])$ is a closed  subset of $BV_r(X,[c,d])$ 
	for every $r \geq d-c$. 
	\item 
	(Analog of Jordan's decomposition)  Every function $f \in BV(X)$  is a difference $f=u-v$ of two order preserving bounded functions $u,v: X \to \R$. 
	\een
\end{lem}
\begin{proof} (1) is clear using the fact that the linear order of  $[c,d]$ is closed. 

(2) Observe that $\Upsilon(f) \leq d-c$ for every increasing function 
$f: X \to [c,d]$.   

(3) 
If in Definition \ref{d:BV} we
allow only the chains $\{x_i\}_{i=1}^n$ with $x_n \leq c$ for some given $c \in X$ then we obtain a
variation on the subset $\{x \in X: x \leq c\} \subset X$. Notation: $\Upsilon^{c}(f)$. 
As in the classical case (as, for example, in \cite{Natanson}) it is easy to see that the functions 
$u(x):=\Upsilon^{x}(f)$ and $v(x):=u(x)-f(x)$ on $X$ are increasing. These functions are bounded because $|\Upsilon^{x}(f)| \leq \Upsilon(f)$ and $f$ is bounded. 

\end{proof}

 \begin{lem} \label{l:FisVectSp} 
$\F(X)$ is a vector space over $\R$ with respect to the natural operations. 
 \end{lem}
 \begin{proof} Clearly, $f \in \F(X)$ implies that $cf \in \F(X)$ for every $c \in \R$. 
 	Let $f_1, f_2 \in \F(X)$. We have to show that $f_1 + f_2 \in \F(X)$. 
 	Let $\emptyset \neq A \subset X$ and $\eps >0$. Since $f_1 \in \F(X)$  there exists an open subset $O_1 \subset X$ such that 
 	$A \cap O_1 \neq \emptyset$ and $f_1(A \cap O_1)$ is $\frac{\eps}{2}$-small. Now since $f_2 \in \F(X)$, for $A \cap O_1$ we can choose an open subset $O_2 \subset X$ such that $(A \cap O_1) \cap O_2$ is nonempty and $f_2(A \cap O_1 \cap O_2)$ is $\frac{\eps}{2}$-small. Then $(f_1+f_2)(A \cap (O_1 \cap O_2))$ is $\eps$-small.   
 	  
 \end{proof}
 
 \begin{cor} \label{c:BVisFr} 
 	$BV(X) \subset \F(X)$ for any LOTS $X$.  
 \end{cor}
 \begin{proof} Any $f \in BV(X)$ is a difference of two increasing functions (Lemma \ref{l:BVprop}.3). Hence we can combine Theorem \ref{monot} and Lemma \ref{l:FisVectSp}. 
 	
 \end{proof}


\begin{thm} \label{newprinciple} 
	For every linearly ordered set $X$ the family of functions $BV_r(X,[c,d])$ 
	is tame.  
	In particular, $M_+(X,[c,d])$ is also tame. 
\end{thm}
\begin{proof}
	Assuming the contrary let $f_n: X \to \R$ be an independent sequence in $BV_r(X,[c,d])$. 
By Lemma \ref{l:BVprop}.3, for every $n$ we have $f_n=u_n -v_n$, where $u_n(x):=\Upsilon^{x}(f_n)$ and $v_n(x):=u_n(x)-f_n(x)$  are increasing functions on $X$. 
Moreover, the family $\{u_n, v_n\}_{n \in \N}$ remains bounded because $|\Upsilon^{x}(f_n)| \leq \Upsilon(f_n) \leq r$ for every $x \in X, n \in \N$ and $f_n$ is bounded. 
	 Apply Representation theorem \ref{RepLem}. Then we conclude that there exist two bounded sequences $t_n: Y \to \R$ and $s_n: Y \to \R$ of continuous increasing functions on a compact LOTS $Y$ which extend $u_n$ and $v_n$. 
	 Consider $F_n:=t_n -s_n$. 
	 First of all note that for sufficiently big $k \in \R$ we have $F_n \in BV_k(Y,[-k,k])$ simultaneously for every $n \in \N$. 
	 
	 Since $F_n|_X=f_n$ 
	 we clearly obtain that the sequence $F_n: Y \to \R$ is independent, too. 
	 On the other hand we can show that  $\cls (\Gamma) \subset \F(Y)$, where $\Gamma=\{F_n\}_{n \in \N} \subset \R^Y$. 
	 Indeed, by	Corollary \ref{c:BVisFr} we know that 
	 $BV_k(Y,[-k,k]) \subset \F(Y)$. Using Lemma \ref{l:BVprop}.1 we get 
	 	$$
	 	\cls(\Gamma) \subset \cls(BV_k(Y,[-k,k])) = BV_k(Y,[-k,k]) \subset \F(Y).
	 	$$
	 	Then $\Gamma$ is a tame family 
	 	by Theorem \ref{f:sub-fr}. This contradiction completes the proof. 
	 	
	 	\end{proof}	

\section{Helly's sequential compactness type theorems} 
\label{s:H}

\begin{thm} \label{t:GenH}  	
	Let $(X, \leq)$ be a linearly ordered set. Then $BV_r(X, [c,d])$ is sequentially compact 
	in the pointwise topology. 
\end{thm} 
\begin{proof} 
First note that 
using Lemma \ref{l:BVprop}.3  
one may reduce the proof to the case where $f_n: X \to \R$ 
is a bounded sequence in $M_+(X)$. 
Now, by Representation theorem \ref{RepLem} we have a bounded sequence of {\it continuous} increasing functions $F_n: Y \to \R$ on a compact LOTS $Y$, where $F_n|_X=f_n$. 
By Theorem \ref{newprinciple} the sequence $F_n$ does not contain an independent subsequence. Hence, by Theorem \ref{f:sub-fr} 
there exists a convergent subsequence $F_{n_k}$. Since the convergence is pointwise and $X$ is a subset of $Y$ we obtain that the corresponding sequence of restrictions $f_{n_k}:=F_{n_k}|_X$ is 
pointwise convergent on $X$. 

\end{proof}

The following corollary can be derived also by results of \cite{FP}. 
Moreover, Theorem \ref{t:GenH} can be proved by \cite[Theorem 7]{FP} using Lemma \ref{l:BVprop}.4. 

\begin{cor} \label{increas} 
	Let $(X, \leq)$ be a linearly ordered set. Then the compact space $M_+(X,[c,d])$ of all order preserving maps 
	is sequentially compact. 
\end{cor} 

Using Nachbin's Lemma \ref{l:Nachbin} we give now in Theorem \ref{corOfGenHelly1} a further generalization  replacing $[c,d]$ in Theorem \ref{t:GenH} by partially ordered compact metrizable spaces. This gives a partial generalization of \cite[Theorem 7]{FP}. Some restriction (e.g., the metrizability) on a compact ordered space $Y$ is really essential as it follows from \cite[Theorem 9]{FP}.

\begin{thm} \label{corOfGenHelly1} 
	Let $(X, \leq)$ be a linearly ordered set and $(Y, \leq)$ be a compact metrizable partially ordered space. Then the compact space $M_+(X,Y)$ of all order preserving maps 
	is sequentially compact. 
\end{thm}
\begin{proof} First of all note that $M_+(X,Y)$ is compact being a closed subset of $Y^X$. Here we have to use the assumption that the given order on $Y$ is closed (Definition \ref{d:ord}). 
	$M_+(X,[0,1])$ is sequentially compact by 
	Theorem \ref{t:GenH}. Therefore, its countable power  $M_+(X,[0,1])^{\N}$ is also sequentially compact. Now observe that $M_+(X,Y)$ is topologically embedded (as a closed subset)  into $M_+(X,[0,1])^{\N}$. 
	Indeed,  by  Nachbin's Lemma \ref{l:Nachbin} continuous increasing maps $Y \to [0,1]$ separate the points. Since $Y$ is a compact metrizable space  one may choose a countable family $h_n$ of increasing continuous maps which separate the points of $Y$. For every $f \in M_+(X,Y)$ define the  function $$u(f): \N \to M_+(X,[0,1]), \ n \mapsto h_n \circ f.$$
	This assignment defines a natural topological embedding of compact Hausdorff spaces (hence, this embedding is closed) 
	$$
	u: M_+(X,Y) \hookrightarrow M_+(X,[0,1])^{\N}, \ \ u \mapsto u(f)=(h_n \circ f)_{n \in \N}.
	$$
	
\end{proof}

Another Helly type theorem can be obtained for functions of bounded variation 
with values into a compact metric space.

\begin{defin} \label{BV2} 
	Let $(Y,d)$ be a metric space and $f: (X, \leq) \to (Y,d)$ be a bounded function. 
	Replacing in Definition \ref{d:BV} the Formula \ref{BV1} by 
	\begin{equation} \label{BVmetrCase}
	\sum_{i=0}^{n-1} d(f(x_{i+1}),f(x_i)) \leq r
	\end{equation}
	We obtain the definition of $f \in BV_r(X,Y)$. 
\end{defin}

In the particular case of $(X,\leq)=[a,b]$ Definition \ref{BV2} and Theorem  \ref{corOfGenHelly2} are well known, \cite{BC,Ch}. 

\begin{thm} \label{corOfGenHelly2} 
	Let $(X, \leq)$ be a linearly ordered set and $(Y,d)$ be a compact metric space. 
	Then the compact space $BV_r(X,Y)$ 
	is sequentially compact in the pointwise topology. 
\end{thm}
\begin{proof} Since $(Y,d)$ is a compact metric space there exist  
	countably many Lipshitz 1 functions $h_n: (Y,d) \to [0,1]$ which separate the points of $Y$. Indeed, take a countable dense subset $\{y_n: \ n \in \N\}$ in $Y$ and define $h_n(x):=d(y_n,x)$. 
	Then $h_n \circ f \in  BV_r(X,[0,1])$ for every $f \in BV_r(X,Y)$. 
	 The rest is similar to the proof of Theorem \ref{corOfGenHelly1}. 
	
\end{proof}

 For a direct proof of Helly's Theorem \ref{t:HellyClassic} see, for example, \cite{Natanson}.  
 An elegant argument was presented by Rosenthal in \cite{Ros1}.  
 The set $M_+([0,1],[0,1])$ is a 
 compact subset in the space $\B_1[0,1]$ of Baire 1 functions. 
 Hence it is sequentially compact because the compactness and sequential compactness are the same for subsets $\B_1(X)$ for any Polish $X$, \cite{Ros1}.  
 
 Another known classical proof is based on the first countability of the Helly's compact space, \cite{StSe}. 
 Such a proof is impossible in general for Corollary \ref{increas}. Indeed, the cardinality $card(M_+(X,[c,d])) \geq card(X)$. When $card(X) > 2^{\aleph_0}$ the corresponding $M_+(X,[c,d])$ is not first countable because 
 the cardinality of a
 first countable compact Hausdorff space is not greater than $2^{\aleph_0}$. 
 However as it was pointed out by Eli Glasner, using a version of Representation theorem \ref{RepLem}, the proof of Corollary \ref{increas} can be reduced to the case when 
 $(X, \leq)$ is metrizable and compact in its interval topology. In this case the principal scheme of the proof in \cite[Exercise 107]{StSe} (as well as the scheme of Rosenthal's argument) 
 seems to be valid with some easy adaptations.

\begin{remarks}
	 As we already mentioned Theorems \ref{t:GenH} and \ref{corOfGenHelly1} cannot be generalized to the assertion with non-separable target spaces. Furthermore, as it was remarked by the referee,  
    a straightforward generalization of Theorem \ref{t:GenH} would be obtained
	if we replace $[c, d]$ by a sequentially compact linearly ordered abelian
	group. 
	The situation with Theorem \ref{corOfGenHelly2}  is similar. 
	In contrast, one may further generalize Theorem \ref{corOfGenHelly1} and obtain the following result: 
	Let $X$ be a linearly ordered set and $Y$ be a sequentially
	compact partially ordered space with hereditary density $< \mathbf{s}$ 
	then $M_+(X,Y)$ is sequentially compact. This assertion can be proved 
	similarly to \cite[Theorem 7]{FP}. It is a generalization of Theorem 5.3 and also a full generalization of \cite[Theorem 7]{FP}. 
\end{remarks}

\nt \textbf{Acknowledgement.} I would like to thank the referee for several important suggestions. 
 
 \vskip 1cm 
 
\bibliographystyle{amsplain}

\end{document}